\documentclass[pdflatex,sn-mathphys-num]{sn-jnl}

\usepackage{graphicx} 
\usepackage{multirow} 
\usepackage{amsmath,amssymb,amsfonts} 
\usepackage{amsthm} 
\usepackage{mathrsfs} 
\usepackage[title]{appendix} 
\usepackage{xcolor} 
\usepackage{textcomp} 
\usepackage{manyfoot} 
\usepackage{booktabs} 
\usepackage{algorithm} 
\usepackage{algorithmicx} 
\usepackage{algpseudocode} 
\usepackage{listings}

\theoremstyle{thmstyleone}%
 \newtheorem{thm}{Theorem}[section]
 \newtheorem{cor}[thm]{Corollary}
 
 \newtheorem{prop}[thm]{Proposition}
 \theoremstyle{definition}
 \newtheorem{defn}[thm]{Definition}
 \theoremstyle{remark}
 
 \newtheorem{ex}[thm]{Example}
 \numberwithin{equation}{section}

\theoremstyle{thmstylethree}%

\begin{document}

\title[Norm Inequalities for Complementable Operators and Parallel Sums]{Norm Inequalities for Complementable Operators and Parallel Sums}

\author*[1]{\fnm{Sachin Manjunath} \sur{Naik}}\email{sachin.n@manipal.edu, sachinmaths46@gmail.com}
\author[2]{\fnm{P. Sam} \sur{Johnson}}\email{sam@nitk.edu.in}

\affil*[1]{\orgdiv{Manipal Institute of Technology}, \orgname{Manipal Academy of Higher Education, Manipal 576 104}, \orgaddress{\street{Karnataka, India}. \city{} \postcode{}}}

\affil[2]{\orgdiv{Department of Mathematical and Computational Sciences}, \orgname{National Institute of Technology Karnataka, Surathkal}, \orgaddress{\street{Mangaluru}, \postcode{575025}, \state{Karnataka}, \country{India}}}


\abstract{This paper investigates the structural and quantitative behaviors of complementable operators on Hilbert spaces, focusing on their norm characteristics and geometric profiles. We establish a comprehensive framework of norm inequalities and lower-bound relationships between a bounded linear operator and its generalized Schur complement (bilateral shorted operator). Under explicit operator factorization and range inclusion criteria, we define the exact conditions under which a bounded linear operator contracts or expands vectors relative to its Schur complement.  Furthermore, we explore the lower boundedness and stability configurations of $(M, N, \lambda)$-complementable operators, proving that a bounded-below Schur complement acts as a sufficient condition to propagate injectivity and lower-bounded stability to the global operator. These structural results are subsequently applied to the network-theoretic setting of the parallel sum of two bounded linear operators. With some specific orthogonality conditions, we derive a novel norm decomposition identity, sharp two-sided global operator bounds, and algebraic restrictions on Douglas reduced solutions via  Moore-Penrose inverses.}

\keywords{Schur complement, complementable operators, parallel sum\\
\noindent\textbf{Mathematics Subject Classification (2020):}
47A08, 47A64}

\maketitle

\section{Introduction}
The mathematical study of operator relationships has been significantly enriched by concepts originally rooted in electrical network theory, most notably the parallel sum and the shorting of matrices. Anderson and Duffin defined the parallel sum $A:B$ for positive semidefinite matrices to model the joint impedance of resistive $n$-port networks \cite{AndersonDuffin}. This operational framework has since undergone extensive generalization. In particular, parallel summation and shorting operations have been successfully extended to non-positive operators to accommodate networks containing reactive elements \cite{Arlinskii}. Parallel to these developments, the study of complementable operators emerged from the classical analysis of the Schur complement. By extending classical matrix decompositions to infinite-dimensional Hilbert spaces, the framework of complementable operators broadens both the theoretical scope and the applicability of these techniques \cite{Antezana}. The classical Schur complement originated in the study of partitioned matrices. 

Let $T$ be a block matrix of size $(n+m) \times (n+m)$, partitioned with respect to the direct sum $\mathbb{C}^n \oplus \mathbb{C}^m$ as
\[
T = \begin{pmatrix} A & B \\ C & D \end{pmatrix},
\]
where $A \in \mathbb{C}^{n \times n}$, $B \in \mathbb{C}^{n \times m}$, $C \in \mathbb{C}^{m \times n}$, and $D \in \mathbb{C}^{m \times m}$. If $D$ is nonsingular, the classical Schur complement of $D$ in $T$, denoted by $T_{/D}$, is defined as
\[
T_{/D} = A - BD^{-1}C.
\]
While this definition plays a fundamental role in matrix analysis across various mathematical and physical domains, it is constrained to an operator of lower order acting on a subspace. 

To overcome this limitation in infinite-dimensional settings, Ando \cite{Ando} generalized the Schur complement by introducing the concept of complementable operators on $\mathbb{C}^n$ with respect to a given subspace $M \subseteq \mathbb{C}^n$. Let $P$ denote the orthogonal projection onto $M$. A matrix $T \in \mathbb{C}^{n \times n}$ is said to be $M$-complementable if there exist matrices $M_r$ and $M_\ell$ satisfying the following conditions:
\[
P M_r = M_r, \quad M_\ell P = M_\ell, \quad P T M_r = P T, \quad \text{and} \quad M_\ell T P = T P.
\]
Under these conditions, the operators $M_\ell T$ and $T M_r$ are uniquely determined and independent of the specific choices of $M_\ell$ and $M_r$. Ando subsequently defined the Schur complement of $T$ with respect to $M$ as
\[
T_{/M} = T - M_\ell T = T - T M_r.
\]
Notably, if $T$ is represented in block form as $\begin{pmatrix} A & B\\ C & D \end{pmatrix}$ with respect to the decomposition $\mathbb{C}^n = M \oplus M^\perp$, and if the range $R(D)$ is closed, then $T_{/M}$ embeds naturally as the block matrix
\[
T_{/M}=\begin{pmatrix} A - B D^{\dagger} C & 0 \\ 0 & 0 \end{pmatrix}.
\]

Ando's foundational framework has inspired numerous modern extensions. Mitra and Puri \cite{Mitra} extended the notion of complementable operators to the abstract setting of operators between distinct Hilbert spaces, thereby introducing and characterizing bilateral shorted operators. For a comprehensive overview of the development and multifaceted applications of the Schur complement, we refer the reader to \cite{Arias, Butler, Carlson1, Carlson3, Krein, Cottle, shorting, Goller, krein3, Ouellette, Zhang}.

Utilizing block matrix representations of operators, Antezana et al. \cite{Antezana} introduced the notions of complementability and weak complementability for bounded linear operators $T: \mathcal{H} \to \mathcal{K}$ with respect to closed subspaces $M \subseteq \mathcal{H}$ and $N \subseteq \mathcal{K}$. Their approach, formulated via range inclusion conditions and operator moduli, led to a rigorous definition of the bilateral shorted operator (Schur complement) in Hilbert spaces while preserving the core algebraic properties of the finite-dimensional case. This bilateral framework proves particularly powerful when defining the parallel sum of bounded linear operators. Specifically, the parallel sum $A:B$ can be realized as the shorted operator of the block matrix 
\[
\begin{pmatrix} A & A \\ A & A+B \end{pmatrix}
\] 
with respect to the closed subspace $\mathcal{H} \oplus \{0\}$.

Further characterizations were established by Arias et al. \cite{Arias}, who investigated the interplay between complementability and the existence of specific bounded projections. While the formulation of the Schur complement in \cite{Antezana} and \cite{Arias} heavily relies on the machinery of weak complementability, Naik and Johnson \cite{IJPAM} recently simplified this framework for complementable operators. They provided a geometric characterization of complementability in terms of the images of the unit ball under the respective block components of $T$. Furthermore, they introduced the notion of $(M,N,\lambda)$-complementable operators and investigated their convergence properties. For additional insights and recent advances in the theory of complementable operators, see \cite{SachinLAA, Massey, SachinAOT, Arias, IJPAM}.

In this paper, we establish several norm inequalities and lower-bound relationships between a bounded linear operator and its Schur complement. Furthermore, we investigate and derive several new properties concerning the operator norm of a parallel sum.

\section{Preliminaries}

The mathematical framework for the study of complementable operators and Schur complements is grounded in the structure of bounded linear operators on Hilbert spaces. Throughout this manuscript, $\mathcal{H}$ and $\mathcal{K}$ denote real or complex Hilbert spaces. We denote $\mathcal{B}(\mathcal{H}, \mathcal{K})$ as the space of all bounded linear operators from $\mathcal{H}$ to $\mathcal{K}$, and let $T^*$, $\mathcal{N}(T)$, and $\mathcal{R}(T)$ represent the adjoint, the null-space, and the range-space of an operator $T$, respectively. For a closed subspace $M \subseteq \mathcal{H}$, the set of vectors within the unit ball is denoted as $\mathcal{B}_M = \{x \in M : \|x\| \leq 1 \}$. A self-adjoint operator $T \in \mathcal{B}(\mathcal{H})$ is defined as positive, denoted $T \geq 0$, if $\langle Tx, x \rangle \geq 0$ for all $x \in \mathcal{H}$. For such positive operators, there exists a unique positive square root denoted by $T^{1/2}$. The modulus of any operator $T$ is defined as $|T| := (T^*T)^{1/2}$. For $T\in \mathcal{B}(\mathcal{H},\mathcal{K})$, the reduced minimum modulus of $T$ is defined as $$\gamma (T)=\inf \{\|Tx\| : x\in \mathcal{H} \cap \mathcal{N}(T){^{\perp}}~ \text{and}~ \|x\|=1 \}.$$  Let $ T: \mathcal{H} \to \mathcal{K} $ be a linear operator and let $ M $ and $ N $ be closed subspaces of $\mathcal{H}$ and $\mathcal{K}$ respectively. We can express $\mathcal{H}$ as the orthogonal direct sum of $ M $ and $ M^{\perp}$ and $\mathcal{K}$ as the orthogonal direct sum of $ N $ and $ N^{\perp}$. With respect to the above decompositions, the operator $ T $ can be written in a block matrix form as
\begin{equation}\label{**}
	T = \begin{pmatrix}
		A & B \\
		C & D
	\end{pmatrix}
\end{equation}
where  $A: M \to N$,
$B: M^{\perp}\to N $,
$C: M \to N^{\perp}$
and $D: M^{\perp}\to N^{\perp}$ are bounded operators.

\begin{thm}(\cite{Douglas})\label{dgls}
	Let $ A, B \in \mathcal{B}(\mathcal{H}) $. Then the following are equivalent:
	\begin{enumerate}
		\item $ \mathcal{R}(A) \subseteq \mathcal{R}(B) $;
		\item $ AA^* \leq \lambda BB^* $, for some $ \lambda > 0 $;
		\item There exists a bounded operator $ C \in \mathcal{B}(\mathcal{H}) $ such that $ A = BC $.
	\end{enumerate}
	Moreover, if these equivalent conditions hold, then there is a unique operator $ C \in \mathcal{B}(\mathcal{H}) $ such that
	\begin{enumerate}
		\item[(i)] $ \|C\| = \inf \{ \lambda > 0 : AA^* \leq \lambda BB^* \} $;
		\item[(ii)] $ \mathcal{N}(A) = \mathcal{N}(B) $;
		\item[(iii)] $ \mathcal{R}(C) \subseteq \mathcal{N}(B)^\perp $.
	\end{enumerate}
\end{thm}
\noindent The unique solution $ C $ is referred as the Douglas reduced solution \cite{redsoln}. An analogous statement applies when $A \in \mathcal{B}(\mathcal{H}_1, \mathcal{H}_2)$ and $B \in \mathcal{B}(\mathcal{H}_3, \mathcal{H}_2)$.

\begin{defn}(\cite{Antezana})
	Let $ P_r \in \mathcal{B}(\mathcal{H}) $ and $ P_\ell \in \mathcal{B}(\mathcal{K}) $ be projections. An operator $ T \in \mathcal{B}(\mathcal{H}, \mathcal{K}) $ is said to be $(P_r, P_\ell)$-complementable if there exist operators $ M_r \in \mathcal{B}(\mathcal{H}) $ and $ M_\ell \in \mathcal{B}(\mathcal{K}) $ such that
	\begin{enumerate}
		\item $ (I-P_r)M_r=M_r $ and $ (I-P_\ell)TM_r=(I-P_\ell) T $;
		\item $ M_\ell (I-P_\ell)=M_\ell $ and $ M_\ell T(I-P_r)=T(I-P_r) $.
	\end{enumerate}
\end{defn}

The above definition is independent of the specific projections considered and it depends solely on the ranges of these projections. The following theorem demonstrates this dependency on the ranges of the projections.

\begin{prop}(\cite{Antezana}) \label{propdef}
	Let $ P_r \in \mathcal{B}(\mathcal{H}) $ and $ P_\ell \in \mathcal{B}(\mathcal{K}) $ be projections whose ranges are $ M $ and $ N $ respectively. Let $ T \in \mathcal{B}(\mathcal{H}, \mathcal{K}) $ be as given in (\ref{**}), the following statements are equivalent:  
	\begin{enumerate}
		\item $ T $ is $(P_r, P_\ell)$-complementable;
		\item $ \mathcal{R}(C) \subseteq \mathcal{R}(D) $ and $ \mathcal{R}(B^*) \subseteq \mathcal{R}(D^*) $;
		\item There exist two projections $ \tilde{P} \in \mathcal{B}(\mathcal{H}) $ and $ \tilde{Q} \in \mathcal{B}(\mathcal{K}) $ such that: $ \mathcal{R}(\tilde{P}^*)=M $, $ \mathcal{R}(\tilde{Q})=N $,  $ \mathcal{R}(T\tilde{P})\subseteq N $ and $ \mathcal{R}((\tilde{Q}T)^*)\subseteq M $.
	\end{enumerate}
\end{prop}

\begin{defn}(\cite{Antezana})
	An operator \quad $ T \in \mathcal{B}(\mathcal{H},  \mathcal{K}) $ \quad  is said to be \quad $(M, N)$-complementable if it is $(P_r, P_\ell)$-complementable for some projections $ P_r $ and $ P_\ell $ with $ \mathcal{R}(P_r)=M $ and $ \mathcal{R}(P_\ell)=N $.
\end{defn}
\begin{thm}(\cite{IJPAM})\label{char2}
		Let  $T \in \mathcal{B}(\mathcal{H},\mathcal{K})$ be as given in (\ref{**}). Then $T$ is $(M,N)$-complementable if and only if there exists $\lambda > 0$ such that  \begin{equation}\label{lambda}
		    C(\mathcal{B}_{M}) \subseteq \lambda D(\mathcal{B}_{ M^{\perp}})~ \text{and}~ B{^*}(\mathcal{B}_{N}) \subseteq \lambda D{^*}(\mathcal{B}_{ N^{\perp}}).
		\end{equation}

        Moreover, if $\mathcal{R}(D)$ is a non-zero closed subspace of $N^{\perp}$, then $T$ is $(M,N)$-\\
        complementable if and only if $$C(\mathcal{B}_{M}) \subseteq \frac{\|C\|}{\gamma (D)}D(\mathcal{B}_{ M^{\perp}}) \quad \text{and} \quad B{^*}(\mathcal{B}_{N}) \subseteq \frac{\|B\|}{\gamma (D{^*})}D{^*}(\mathcal{B}_{N^{\perp}})$$
	\end{thm}
 \noindent If (\ref{lambda}) holds, then $T$ is said to be $(M,N, \lambda)$-complementable\cite{SachinLAA}. 
 
\begin{defn}(\cite{Antezana})
An operator $ T \in \mathcal{B}(\mathcal{H}, \mathcal{K}) $ of the from (\ref{**}) is said to be $(M,N)$-weakly complementable if
	 $$ \mathcal{R}(C) \subseteq \mathcal{R}(|D^*|^{\frac{1}{2}}) \ \text{and} \ \mathcal{R}(B^*) \subseteq \mathcal{R}(|D|^{\frac{1}{2}}).$$
\end{defn}
It is easy to see that if $ T $ is $(M,N)$-complementable, then it is $(M,N)$-weakly complementable. Both notions of complementability coincide if $ \mathcal{R}(D) $ is closed.

\begin{defn}(\cite{IJPAM})
For a weakly complementable operator $T$, the Schur complement (or bilateral shorted operator) is defined as $T_{/(M,N)} = \begin{pmatrix} A - F^*E & 0 \\ 0 & 0 \end{pmatrix}$. 
\end{defn}

For $(M, N)$-complementable operators, this representation simplifies significantly, allowing the Schur complement to be written as $$T_{/(M,N)} = \begin{pmatrix} A - BZ & 0 \\ 0 & 0 \end{pmatrix} = \begin{pmatrix} A - YC & 0 \\ 0 & 0 \end{pmatrix},$$ where $Z$ and $Y^*$ are the Douglas reduced solutions of $\mathcal{R}(C) \subseteq \mathcal{R}(D)$ and $\mathcal{R}(B^*) \subseteq \mathcal{R}(D^*)$, respectively.

\begin{thm}(\cite{IJPAM})\label{eqgm}
Let $T \in \mathcal{B}(\mathcal{H},\mathcal{K})$ with $\mathcal{R}(D)$ closed. Then $T$ is $(M,N)$-complementable if and only if for each $x \in M$, there exists unique $z\in M$ such that $\{T(x,0)+T(M{^{\perp}})\} \cap N = \{(z,0)\}$. Moreover, $T_{/(M,N)}(x,0)=(z,0)$.
	\end{thm}

\begin{defn}(\cite{Antezana})
Let $A, B \in L(\mathcal{H}_1, \mathcal{H}_2)$. $A$ and $B$ are said to be parallel summable $\mathcal{R}(A) \subseteq \mathcal{R}(A+B)$ and $\mathcal{R}(A^*) \subseteq \mathcal{R}(A^*+B^*).$

In this case, the parallel sum of $A$ and $B$, denoted by $A : B \in L(\mathcal{H}_1, \mathcal{H}_2)$, is
\[
\begin{pmatrix} A : B & 0 \\ 0 & 0 \end{pmatrix} = \begin{pmatrix} A & A \\ A & A + B \end{pmatrix}_{/(\mathcal{H}_1 \oplus \{0\}, \mathcal{H}_2 \oplus \{0\})}.
\]
\end{defn}

The above definition is equavalent to say that the operator $\begin{pmatrix} A & A \\ A & A + B \end{pmatrix}$ is $(\mathcal{H}_1 \oplus \{0\}, \mathcal{H}_2 \oplus \{0\})$-complementable.

\section{Norm Inequalities and Comparisons}

This section explores the core structural and norm properties of complementable operators. Under standard range inclusion conditions, which guarantee the existence of a Douglas solution to operator equations, we analyse how the operator $T$ contracts or expands vectors compared to its Schur complement $T_{/(M,N)}$. Our first main result establishes equivalent conditions for these norm inequalities in terms of positive semidefinite operator relations involving the block components of $T$.

Throughout this section, we assume that $M$ and $N$ are closed subspaces of $\mathcal{H}_1$ and $\mathcal{H}_2$, respectively, and that $T \in \mathcal{B}(\mathcal{H}_1, \mathcal{H}_2)$ is given in the form displayed in equation (\ref{**}).
\begin{thm}\label{3.1}
 Let $T \in \mathcal{B}(\mathcal{H}_1, \mathcal{H}_2)$ be as specified in equation (\ref{**}). Assume that $T$ is $(M,N)$ is complementable and that $X$ is an operator satisfying $C = DX$. Then, for every $x \in M$, we have:
\begin{enumerate}
    \item $\|Tx\| \geq \|T_{/(M,N)}x\| \iff C^*C + X^*B^*A + A^*BX - X^*B^*BX \geq 0 \iff C^*C + A^*A \geq (T_{/(M,N)})^*(T_{/(M,N)})$.
    \item $\|Tx\| \leq \|T_{/(M,N)}x\| \iff C^*C + X^*B^*A + A^*BX - X^*B^*BX \leq 0 \iff C^*C + A^*A \leq (T_{/(M,N)})^*(T_{/(M,N)})$.
    \item $\|Tx\| = \|T_{/(M,N)}x\| \iff C^*C + X^*B^*A + A^*BX - X^*B^*BX = 0 \iff C^*C + A^*A = (T_{/(M,N)})^*(T_{/(M,N)})$.
\end{enumerate}
\end{thm}

\begin{proof}
When $T$ acts on $x \in M$, we obtain: 
$$T \begin{pmatrix} x \\ 0 \end{pmatrix} = \begin{pmatrix} A & B \\ C & D \end{pmatrix} \begin{pmatrix} x \\ 0 \end{pmatrix} = \begin{pmatrix} Ax \\ Cx \end{pmatrix}.$$
The squared norm is then the sum of the squared norms of these components:
$$\|Tx\|^2 = \|Ax\|^2 + \|Cx\|^2 = \langle A^*Ax, x \rangle + \langle C^*Cx, x \rangle = \langle (A^*A + C^*C)x, x \rangle.$$
Thus, the $(1,1)$ block of the operator $T^*T$ restricted to $M$ is given by $(A^*A + C^*C)$.

For the Schur complement $T_{/(M,N)}$, the associated squared norm on the subspace $M$ is:
$$\|T_{/(M,N)}x\|^2 = \|(A - BX)x\|^2.$$
Writing this as an inner product yields:
$$\|T_{/(M,N)}x\|^2 = \langle (A - BX)^*(A - BX)x, x \rangle.$$
Expanding the operator factor gives:
$$(A - BX)^*(A - BX) = (A^* - X^*B^*)(A - BX)$$
$$= A^*A - A^*BX - X^*B^*A + X^*B^*BX.$$
Consider the difference between the squared norms:
\begin{align*}
    \|Tx\|^2 - \|T_{/(M,N)}x\|^2 &= \langle (A^*A + C^*C)x, x \rangle\\&~~ - \langle (A^*A - A^*BX - X^*B^*A + X^*B^*BX)x, x \rangle\\
    &= \langle (C^*C + X^*B^*A + A^*BX - X^*B^*BX)x, x \rangle.
\end{align*}
We can then express
\begin{align*}
C^*C + X^*B^*A + A^*BX - X^*B^*BX &= C^*C + X^*B^*(A - BX) + A^*BX \\
&= C^*C + X^*B^*(A - BX) \\
&~~~- A^*(A - BX) + A^*A \\
&= C^*C - (A^* - X^*B^*)(A - BX) + A^*A \\
&= C^*C + A^*A - (T_{/(M,N)})^*(T_{/(M,N)}).
\end{align*}
From this identity, the following equivalences for item (1) are obtained:
\begin{itemize}
    \item $\|Tx\| \geq \|T_{/(M,N)}x\| \iff \|Tx\|^2 - \|T_{/(M,N)}x\|^2 \geq 0 \iff T^*T \geq (T_{/(M,N)})^*(T_{/(M,N)})$.
    \item Using the algebraic expansion above, this is equivalent to
    \[ C^*C + X^*B^*A + A^*BX - X^*B^*BX \geq 0. \]
    \item It is further equivalent to
    \[ C^*C + A^*A \geq (T_{/(M,N)})^*(T_{/(M,N)}). \]
\end{itemize}
Items (2) and (3) are obtained by replacing the inequality with $\leq$ and $=$, respectively.
\end{proof}

Following the characterization of the operator norms on the subspace $M$, we turn our attention to the dual behavior involving the adjoints of these operators. By utilizing the explicit factorization $C = DX$ guaranteed by the complementability conditions, we can establish an analogue for the actions of $T^*$ and $(T_{/(M,N)})^*$ on the subspace $N$. The next theorem provides the exact algebraic and operator-order conditions under which the norm of the adjoint operator dominates, is dominated by, or is identical to the norm of the adjoint of its Schur complement.

\begin{thm}
 Let $T \in \mathcal{B}(\mathcal{H}_1, \mathcal{H}_2)$ be as specified in equation (\ref{**}). Assume that $T$ is $(M,N)$ is complementable and that $X$ is an operator satisfying $C = DX$. Then, for every $x \in N$, we have:
 
\begin{enumerate}
    \item $\|T^* x\| \geq \|(T_{/(M,N)})^* x\| \iff BB^* + AX^* B^* + BXA^* - BXX^* B^* \geq 0 \iff BB^* + AA^* \geq (T_{/(M,N)})(T_{/(M,N)})^*.$
    \item $\|T^* x\| \leq \|(T_{/(M,N)})^* x\| \iff BB^* + AX^* B^* + BXA^* - BXX^* B^* \leq 0 \iff BB^* + AA^* \leq (T_{/(M,N)})(T_{/(M,N)})^*.$
    \item $\|T^* x\| = \|(T_{/(M,N)})^* x\| \iff BB^* + AX^* B^* + BXA^* - BXX^* B^* = 0 \iff BB^* + AA^* = (T_{/(M,N)})(T_{/(M,N)})^*.$
\end{enumerate}
\end{thm}

It is worth noting that while the Schur complement is often viewed as a ``reduction" of an operator, its norm (and the norm of its adjoint) can strictly exceed that of the original operator. The following example provides a numerical validation of this phenomenon, explicitly verifying the operator inequalities established in Theorem \ref{3.1} for a specific subspace $M$.

\begin{ex}
Let $T = \begin{pmatrix} 0 & 2 \\ 1 & \frac{1}{2} \end{pmatrix} : \mathbb{R}^2 \to \mathbb{R}^2$ and $M = \{(x,0) : x \in \mathbb{R}\}$. 
Suppose $(T_{/(M,N)}) = \begin{pmatrix} -4 & 0 \\ 0 & 0 \end{pmatrix}$. 
In this case, $\|T_{/(M,N)}\| = 4$ and $\|T\| \approx 2.08$. 
With $X = 2$, we calculate:
\[ BB^* + AX^* B^* + BXA^* - BXX^* B^* = 4 - 8 = -4 \leq 0. \]
This consistent result supports the inequality $\|T^* x\| \leq \|(T_{/(M,N)})^* x\|$.
\end{ex}

Building upon the geometric ball-inclusion estimates, we now establish sharp global operator norm bounds utilizing the explicit framework of $(M,N,\lambda)$-complementable operators. The following theorem demonstrates how the structural scaling parameter $\lambda$ directly governs the magnitude of the Schur complement $T_{/(M,N)}$ relative to both the unrestricted operator norm and its restriction to the specific subspace $L$.

\begin{thm}
 Let $T \in \mathcal{B}(\mathcal{H}_1, \mathcal{H}_2)$ be as specified in equation (\ref{**}). If $T$ is $(M,N,\lambda)$-complementable, then
\[ \|{T_{/(M,N)}} \|\leq \sqrt{1 + \lambda^2} \|{T|_L}\| \quad \text{and} \quad \|{T_{/(M,N)}}\| \leq \sqrt{1 + \lambda^2} \|{T}\| \]
where $L = \{(x, y) : Cx + Dy = 0\}$.
\end{thm}

\begin{proof}
By the definition of the Schur complement operator norm, we have:
\begin{align*}
\|{T_{/(M,N)}}\| &= \sup \{ \|{T_{/(M,N)}(x, 0)}\| : x \in M, \|{x}\| \leq 1 \} \\
&= \sup \{ \|{T(x, -Xx)} \|: x \in M, \|{x}\| \leq 1 \} \\
&= \sup \{ \|{Ty}\| : y \in L, \|{y}\| \leq \sqrt{1 + \|{X}^2}\| \} \\
&= (\sqrt{1 + \|{X}^2}\|) \sup \{ \|{Ty}\| : y \in L, \|{y}\| \leq 1 \} \\
&= \sqrt{1 + \|{X}^2}\| \|{T|_L\|}.
\end{align*}

Let $x \in M$ be such that $\|x\| \leq 1$. Since $T$ is $(M,N,\lambda)$-complementable, Theorem \ref{char2} guarantees the existence of some $y \in M^\perp$ with $\|y\| \leq 1$ satisfying $Cx = \lambda Dy$. We can decompose $y$ as $y = y_1 + y_2$, where $y_1 \in \mathcal{N}(D)^\perp \cap M^\perp$ and $y_2 \in \mathcal{N}(D) \cap M^\perp$. 

Thus, we have $\|y_1\| \leq \|y\| \leq 1$ and $Cx = \lambda D y_1$. We now introduce the operator $X : M \to M^\perp$ by setting $Xx = \lambda y_1$, which implies $C = DX$. The norm of $X$ can be estimated as follows:
\[ \|{X}\| = \sup \{ \|{Xx}\| : x \in M, \|{x}\| \leq 1 \} \leq \sup \{ \lambda \|{y}\| : y \in M^\perp, \|{y}\| \leq 1 \} \leq \lambda. \]
Substituting this bound into our earlier norm equality, we obtain:
\[ \|{T_{/(M,N)}}\| \leq \sqrt{1 + \lambda^2} \|{T|_L}\|. \]
Finally, since $\|{T|_L} \|\leq\| {T}\|$, we conclude:
\[ \|{T_{/(M,N)}}\| \leq \sqrt{1 + \lambda^2} \|{T}\|. \]
\end{proof}

While the preceding results characterize norm inequalities via algebraic operator relations, it is highly instructive to consider cases where these relations stem from pure geometric configurations of the underlying subspaces. The following theorem demonstrates how orthogonality between the actions of $T$ on $M$ and $M^\perp$ acts as a definitive toggle, completely dictating whether the Schur complement compresses or expands the vector norms on $M$.

\begin{thm}
 Let $T \in \mathcal{B}(\mathcal{H}_1, \mathcal{H}_2)$ be as specified in equation (\ref{**}). Assume that $T$ is $(M,N)$ is complementable.

\begin{enumerate}
    \item If $T(M^\perp) \perp M$, then $\|T_{/(M,N)}x\| \leq \|Tx\|$ for all $x \in M$.
    \item If $T(M) \perp T(M^\perp)$, then $\|T_{/(M,N)}x\| \geq \|Tx\|$ for all $x \in M$.
\end{enumerate}
\end{thm}

\begin{proof}
Assume first that $T(M^\perp) \perp M$. For any $x \in M$, we can write the Schur complement as:
\[ T_{/(M,N)}x = Tx + Ty, \quad \text{for some } y \in M^\perp. \]
Rearranging this, we have $Tx = T_{/(M,N)}x - Ty$. Since $T_{/(M,N)}x \in M$ and $Ty \in T(M^\perp)$, the assumption $T(M^\perp) \perp M$ implies that $T_{/(M,N)}x$ and $Ty$ are orthogonal. By the Pythagorean theorem:
\[ \|Tx\|^2 = \|T_{/(M,N)}x\|^2 + \|Ty\|^2 \geq \|T_{/(M,N)}x\|^2. \]
Thus, $\|T_{/(M,N)}x\| \leq \|Tx\|$.

\medskip

Conversely, suppose that $T(M) \perp T(M^\perp)$. For any $x \in M$, we can write $T_{/(M,N)}x = Tx + Ty$ with some $y \in M^\perp$. Because $x$ lies in $M$, we obtain $Tx \in T(M)$. By assumption $T(M) \perp T(M^\perp)$, we know that $Tx$ and $Ty$ are orthogonal. Applying the Pythagorean theorem to this sum:
\[ \|T_{/(M,N)}x\|^2 = \|Tx + Ty\|^2 = \|Tx\|^2 + \|Ty\|^2 \geq \|Tx\|^2. \]
Hence, $$\|T_{/(M,N)}x\| \geq \|Tx\|.$$
\end{proof}

Our next result explores the lower boundedness and stability profiles of $(M,N,\lambda)$-complementable operators. We show that a bounded-from-below Schur complement acts as a sufficient condition to guarantee that the global operator $T$ is also bounded below. Crucially, we quantify this link by providing a sharp lower estimate for vectors restricted to the closed subspace $M$.

\begin{thm}\label{3.6}
 Let $T \in \mathcal{B}(\mathcal{H}_1, \mathcal{H}_2)$ be as specified in equation (\ref{**}). Assume that $T$ is $(M,N,\lambda)$-complementable. If the Schur complement $T_{/(M,N)}$ is bounded below, then $T$ is bounded below on $M$ and 
\[ \frac{\gamma (T_{/(M,N)})}{1+\lambda} \|x\| \leq \|Tx\| \quad \text{for all } x \in M. \]
\end{thm}

\begin{proof}
The Schur complement can be written in the form $T_{/(M,N)} = M_l T$, where $M_l = \begin{pmatrix} I & -Y \\ 0 & 0 \end{pmatrix}$ and $Y$ is an operator satisfying $B = YD$.

If $T_{/(M,N)}$ is bounded below on $M$, we have $\gamma (T_{/(M,N)})> 0$ and for all $x \in M$:
\[ \gamma (T_{/(M,N)}) \|x\| \leq \|T_{/(M,N)}x\| = \|M_l Tx\| \leq \|M_l\| \|Tx\|. \]
Note that $\|M_l\| \leq 1 + \|Y\|$. Substituting this into the inequality, we obtain:
\[ \gamma (T_{/(M,N)}) \|x\| \leq (1 + \|Y\|) \|Tx\| \text{ for all } x \in M\] Thus,\[\frac{\gamma (T_{/(M,N)})}{1 + \|Y\|} \|x\| \leq \|Tx\| \text{ for all } x \in M. \]
This proves that $T$ is bounded below on $M$.

Now, since $T$ is $(M,N,\lambda)$-complementable, we have $\|Y\| = \|Y^*\| \leq \lambda$. Consequently, the inequality becomes:
\[ \frac{\gamma (T_{/(M,N)})}{1 + \lambda} \|x\| \leq \|Tx\|  \text{ for all } x \in M. \]
\end{proof}

As an immediate consequence of the preceding lower-boundedness theorem, we obtain a precise result regarding the injectivity and invertibility of the global operator $T$. When the Schur complement $T_{/(M,N)}$ possesses a bounded inverse on its range, this structural property propagates to the restriction of $T$ on the subspace $M$. The following corollary establishes this injectivity and delivers an explicit upper bound for the norm of the corresponding inverse operator.

\begin{cor}\label{3.7}
Let $T \in \mathcal{B}(\mathcal{H}_1, \mathcal{H}_2)$ be as specified in equation (\ref{**}). Assume that $T$ is $(M,N,\lambda)$-complementable.
If $T_{/(M,N)}$ is bounded below  
then $T$ is injective on $M$ and  inverse of $T$ restricted to $M$ satisfies:
\[ \|(T_{|_M})^{-1}\| \leq \frac{1 + \|Y\|}{\gamma (T_{/(M,N)})} \leq \frac{1 + \lambda}{\gamma (T_{/(M,N)})}. \]
\end{cor}
\begin{proof}
    Since $T_{/(M,N)}$ is bounded below, by Theorem \ref{3.6}, we have \[ \frac{\gamma (T_{/(M,N)})}{1 + \lambda} \|x\| \leq \|Tx\| \text{ for all } x 
    \in M.\]
    Also, $T_{/(M,N)}$ is bounded below implies $\gamma (T_{/(M,N)})>0.$ Therefore, we have $\frac{\gamma (T_{/(M,N)})}{1 + \lambda}>0.$ Thus $ (T_{|_M})$ is injective and
    $$\|(T_{|_M})^{-1}\| \leq \frac{1 + \|Y\|}{\gamma (T_{/(M,N)})} \leq \frac{1 + \lambda}{\gamma (T_{/(M,N)})}.$$
\end{proof}

\section{Applications to the Parallel Sum of Operators}

Having established a comprehensive framework of norm inequalities and stability bounds for generalized Schur complements, we now direct our focus to the quantitative behavior of the parallel sum of bounded linear operators. 

As introduced by Antezana et al. \cite{Antezana}, the parallel sum $A:B$ of two bounded linear operators $A$ and $B$ can be elegantly realized via the shorting (or bilateral Schur complementation) of a specialized $2 \times 2$ block operator matrix with respect to a canonical subspace. Consequently, the structural and geometric theorems derived in the previous section naturally yield profound, sharp insights when projected onto this network-theoretic setting. 

Our first primary result in this direction establishes an exact, non-trivial norm decomposition identity for the parallel sum. By identifying a specific condition among the component operators, we demonstrate that the norm of the parallel sum can be written as a direct sum of orthogonal actions of the global block operator. 
In this section, we establish a novel norm inequality for the parallel sum of operators by constructing a specific block operator matrix and utilizing the orthogonality conditions derived from the $(M,N)$-complementability framework.

Consistent with the equivalences detailed in Section 2, we restate the parallel summability of $A:\mathcal{H}_1 \to \mathcal{H}_2$ and $B:\mathcal{H}_1 \to \mathcal{H}_2$ in terms of the $(\mathcal{H}_1 \oplus \{0\}, \mathcal{H}_2 \oplus \{0\})$-complementability of the block matrix $\begin{pmatrix} A & A \\ A & A + B \end{pmatrix}$, which serves as our primary working definition for the remainder of this section. Throughout this section, we work under the assumption that 
\[ T = \begin{pmatrix} A & A \\ A & A+B \end{pmatrix} \in \mathcal{B}(\mathcal{H}_1 \oplus \mathcal{H}_1, \mathcal{H}_2 \oplus \mathcal{H}_2), \]
and that $T$ is $(M,N)$-complementable, where $M = \mathcal{H}_1 \oplus \{0\}$ and $N = \mathcal{H}_2 \oplus \{0\}.$
\begin{thm}\label{4.1}
Let $T = \begin{pmatrix} A & A \\ A & A+B \end{pmatrix} \in \mathcal{B}(\mathcal{H}_1 \oplus \mathcal{H}_1, \mathcal{H}_2 \oplus \mathcal{H}_2)$ be $(M,N)$-complementable. If the operator condition $A^*(2A + B) = 0$ is satisfied, then for any $x \in M$, the parallel sum $A:B$ satisfies the following norm identity:
\[ \|(A:B)x\|  \geq \|Tx\| .\]
\end{thm}

\begin{proof}
Consider $x \in M$ and $y \in M^\perp$, represented as vectors in the direct sum space:
\[ x = \begin{pmatrix} x_1 \\ 0 \end{pmatrix}, \quad y = \begin{pmatrix} 0 \\ y_2 \end{pmatrix} \quad \text{where } x_1, y_2 \in \mathcal{H} \]
The action of the block operator $T$ on these subspaces is given by:
\[  Tx = \begin{pmatrix} Ax_1 \\ Ax_1 \end{pmatrix}, \quad  Ty = \begin{pmatrix} Ay_2 \\ (A+B)y_2 \end{pmatrix} \]

\noindent The inner product of the images $Tx \in T(M)$ and $Ty \in T(M^\perp)$ is calculated as follows:
\begin{align*}
    \langle  Tx,  Ty \rangle &= \langle Ax_1, Ay_2 \rangle + \langle Ax_1, (A+B)y_2 \rangle \\
    &= \langle x_1, A^*Ay_2 \rangle + \langle x_1, A^*(A+B)y_2 \rangle\\
    &= \langle x_1, (2A^*A + A^*B)y_2 \rangle\\ &= \langle x_1, A^*(2A + B)y_2 \rangle
\end{align*}
Under the hypothesized condition $A^*(2A + B) = 0$, the inner product vanishes for all $x_1, y_2 \in \mathcal{H}$, implying that $T(M) \perp T(M^\perp)$.

From Theorem \ref{eqgm}, for any $x \in M$, the Schur complement can be decomposed as $T_{/(M,N)}x = Tx + Ty$ for some $y \in M^\perp$. Given the established orthogonality $T(M) \perp T(M^\perp)$, we apply the Pythagorean theorem:
\[ \|T_{/(M,N)}x\|^2 = \|Tx\|^2 + \|Ty\|^2 \]
By substituting the parallel sum $A:B$ for the Schur complement $T_{/(M,N)}$, we obtain the desired identity:
\[ \|(A:B)x\|^2 = \|Tx\|^2 + \|Ty\|^2 \geq \|Tx\|^2 \]

\end{proof}

As established in the preceding proof, the operator condition $A^*(2A + B) = 0$ is sufficient to guaranty $T(M) \perp T(M^\perp)$. Importantly, the converse of this statement is also true, establishing a strict equivalence.

Assume $T(M) \perp T(M^\perp)$. For any $x = \begin{pmatrix} x_1 \\ 0 \end{pmatrix} \in M$ and $y = \begin{pmatrix} 0 \\ y_2 \end{pmatrix} \in M^\perp$, we have
\begin{align*}
    0&=\langle Tx, Ty \rangle\\
    &=\langle Ax_1, Ay_2 \rangle + \langle Ax_1, (A+B)y_2 \rangle\\
    &=\langle x_1, A^*Ay_2 \rangle + \langle x_1, A^*(A+B)y_2 \rangle\\
    &=\langle x_1, A^*(2A+B)y_2 \rangle
\end{align*}
Since this inner product vanishes for all  $x_1, y_2 \in \mathcal{H}_1$, we deduce that $A^*(2A + B) = 0$.\\
Consequently, if $A^*(2k + B) = 0$ with $k \neq 2$, it does not necessarily follow that $\|(A:B)x\| \geq \|Tx\|.$

In addition to exact norm identities, it is highly valuable to establish sharp, two-sided inequalities that bound the global block operator $T$ using the intrinsic spectral and structural metrics of its components. By utilizing a Douglas solution $Y$ satisfying $A = Y(A+B)$, the following proposition provides a comprehensive estimate for $\|T\|$. The lower bound is governed by the reduced minimum modulus of the parallel sum, $\gamma(A:B)$, while the upper bound scales directly with the norm of the joint operator $A+B$.

\begin{prop}
Let $A, B \in \mathcal{B}(\mathcal{H}_1, \mathcal{H}_2)$ be operators such that $$T = \begin{pmatrix} A & A \\ A & A+B \end{pmatrix} \in \mathcal{B}(\mathcal{H}_1 \oplus \mathcal{H}_1, \mathcal{H}_2 \oplus \mathcal{H}_2)$$ is $(M,N,\lambda)$-complementable. Then $$\|T\| \leq \sqrt{3\lambda^2 + 1} \|A+B\|.$$ Moreover, if $A^*(2A+B)=0$ and $A:B, B$ are bounded below, then
\begin{equation*}
    \min\Big\{\frac{\gamma(A:B)}{1+\lambda},\frac{\gamma(B)}{2} \Big\}\leq \|T\| \leq \sqrt{3\lambda^2 + 1} \|A+B\|.
\end{equation*}

\end{prop}

\begin{proof}
\textbf{1. Derivation of the Upper Bound:}
The block operator $T$ can be factorized as the product of a coefficient matrix and a diagonal block operator:
\[ T = \begin{pmatrix} A & A \\ A & A+B \end{pmatrix} = \begin{pmatrix} Y(A+B) & Y(A+B) \\ Y(A+B) & A+B \end{pmatrix} = \begin{pmatrix} Y & Y \\ Y & I \end{pmatrix} \begin{pmatrix} A+B & 0 \\ 0 & A+B \end{pmatrix}, \] where $Y$ is the reduced Douglas solution to $A=X(A+B).$\\
By the sub-multiplicative property of the operator norm, we have:
\[ \|T\| \leq \left\| \begin{pmatrix} Y & Y \\ Y & I \end{pmatrix} \right\| \cdot \left\| \begin{pmatrix} A+B & 0 \\ 0 & A+B \end{pmatrix} \right\| \]
The norm of the diagonal block matrix is $\|A+B\|$. For the first factor, we apply the standard block matrix norm inequality $\| \begin{pmatrix} X_{11} & X_{12} \\ X_{21} & X_{22} \end{pmatrix} \| \leq \sqrt{\|X_{11}\|^2 + \|X_{12}\|^2 + \|X_{21}\|^2 + \|X_{22}\|^2}$:
\[ \left\| \begin{pmatrix} Y & Y \\ Y & I \end{pmatrix} \right\| \leq \sqrt{\|Y\|^2 + \|Y\|^2 + \|Y\|^2 + \|I\|^2} = \sqrt{3\|Y\|^2 + 1} .\]
$Y$ is reduced Douglas solution to $A=X(A+B)$ and $T$ is $(M,N,\lambda)$-compleme-
ntable, we have $\|Y\| \leq \lambda.$
Substituting $\|Y\| \leq \lambda$ yields the upper bound \begin{equation}\label{eq4.1}
    \|T\| \leq \sqrt{3\lambda^2 + 1} \|A+B\|.
\end{equation}

\noindent \textbf{2. Derivation of the Lower Bound:}
From the Theorem \ref{3.6}, we have  \[ \frac{\gamma (T_{/(M,N)})}{1 + \lambda} \|x\| \leq \|Tx\|  \text{ for all } x \in M. \]

Hence \[ \frac{\gamma (A:B)}{1 + \lambda}  \leq \|T_{|M}\|. \]
With respect to the decomposition, $\mathcal{H}_1=M^{\perp} \oplus (M^{\perp})^{\perp}$ and $\mathcal{H}_2=N^{\perp}\oplus (N^{\perp})^{\perp},$ the operator $T$ has the decomposition $T=\begin{pmatrix}
    A+B & A\\
    A & A
\end{pmatrix}.$ Clearly $T$ is $(M^{\perp}, N^{\perp})$-complementable and $T_{/(M^{\perp},N^{\perp})}=\begin{pmatrix}
    B & 0\\
    0 & 0
\end{pmatrix}.$
Again by Theorem \ref{3.6}, we have  \[ \frac{\gamma (T_{/(M^{\perp},N^{\perp})})}{1 + 1} \|x\| \leq \|Tx\|  \text{ for all } x \in M^{\perp}. \]

\noindent Hence \[ \frac{\gamma (B)}{2}  \leq \|T_{|{M^{\perp}}}\|. \]

Let $\gamma= \min\Big\{\frac{\gamma (B)}{2}, \frac{\gamma (A:B)}{1 + \lambda}\Big\}.$ Clearly $\gamma>0,$\\
Now for any $z=x \oplus y \in \mathcal{H}_1\oplus \mathcal{H}_1,$ we have $Tz=Tx+Ty.$ Since $A^*(2A+B)=0,$ from the proof of Theorem \ref{4.1} we have noted that $T(M) \perp T(M^{\perp}).$ Therefore, $Tx\perp Ty.$ Hence, \begin{align*}
    \|Tz\|^2&=\|Tx\|^2+\|Ty\|^2\\
     & \geq \Big(\frac{\gamma (A:B)}{1 + \lambda}\Big)^2\|x\|^2+ \Big(\frac{\gamma (B)}{2}\Big)^2 \|y\|^2\\
     &\geq \gamma^2 (\|x\|^2+\|y\|^2)=\gamma ^2 \|z\|^2.
\end{align*}
This gives $$\gamma \|z\|\leq \|Tz\| \text{ for all } z\in \mathcal{H}_1 \oplus \mathcal{H}_2.$$

\noindent Thus $$\gamma = \min\Big\{\frac{\gamma (B)}{2}, \frac{\gamma (A:B)}{1 + \lambda}\Big\} \leq \|T\|.$$
Combining with equation \ref{eq4.1},
$$\gamma = \min\Big\{\frac{\gamma (B)}{2}, \frac{\gamma (A:B)}{1 + \lambda}\Big\} \leq \|T\|\leq \sqrt{3\lambda^2+1}\|A+B\|.$$
This completes the proof.
\end{proof}

In the next proposition, we explore how the structural condition $A^*(2A+B)=0$ impacts the relative magnitudes of the parallel sum $A:B$ and its constituent operator $A$. We establish that this condition prevents the parallel sum from acting as a strict contraction relative to $A$, enforcing a scaling factor of $\sqrt{2}$. Furthermore, we demonstrate that this geometric layout induces a rigid lower threshold for the complementability index $\lambda$.

\begin{prop}
Let $A, B \in \mathcal{B}(\mathcal{H}_1, \mathcal{H}_2)$ be operators such that $$T = \begin{pmatrix} A & A \\ A & A+B \end{pmatrix} \in \mathcal{B}(\mathcal{H}_1 \oplus \mathcal{H}_1, \mathcal{H}_2 \oplus \mathcal{H}_2)$$ is $(M,N)$-complementable. If  $A^*(2A+B)=0$ , then \begin{enumerate}
    \item $\|(A:B)x\|^2 \ge 2\|Ax\|^2$ for all $x \in M$
    \item $\|A:B\| \ge \sqrt{2}\|A\|.$
\end{enumerate} 
Moreover, if $A^*(2A+B)=0$ and $T$ is $(M,N,\lambda)$-complementable, then $$\lambda \geq \sqrt{2}-1.$$
\end{prop}

\begin{proof}
Consider $x \in M$ represented as $(x_1, 0)^T$ with $x_1 \in \mathcal{H}_1$. The action of the block operator $T$ on $M$ is given by $ Tx = (Ax_1, Ax_1)^T$. The squared norm of this image is:
\begin{equation*}
\|Tx\|^2 = \|Ax_1\|^2 + \|Ax_1\|^2 = 2\|Ax_1\|^2 \text{.}
\end{equation*}
The condition $A^*(2A+B)=0$ implies that the inner product of the images $T(M)$ and $T(M^\perp)$ vanishes for all $x_1, y_2 \in \mathcal{H}$. Thus, $T(M) \perp T(M^\perp)$. 

By Theorem \ref{eqgm}, for complementable operators, the Schur complement (parallel sum) can be decomposed as $(A:B)x = Tx + Ty$ for some $y \in M^\perp$. Due to the established orthogonality, we apply the Pythagorean theorem:
\begin{equation*}
\|(A:B)x\|^2 = \|Tx\|^2 + \|Ty\|^2 \text{.}
\end{equation*}
Substituting the calculation for $\|Tx\|^2$ and noting $\|Ty\|^2 \ge 0$, we obtain $\|(A:B)x\|^2 \ge 2\|Ax\|^2$. Taking the supremum over all unit vectors $x \in \mathcal{H}$ on both sides yields $\|A:B\|^2 \ge 2\|A\|^2$.

Now suppose, $T$ is $(M,N,\lambda)$-complementable and $A^*(2A+B)=0.$ Using the definition of the parallel sum, we can write $\begin{pmatrix}
    A{:}B & 0\\
    0&0
\end{pmatrix}{=}\begin{pmatrix}
    A-AZ &0\\
    0 &0
\end{pmatrix}.$ Therefore, $$\|(A:B)\|=\|A(I_M-Z)\|\leq \|A\|\|I_M -Z\|=\|A\|(1+\lambda).$$  Already, we have established \begin{equation*}
    \|A:B\| \ge \sqrt{2}\|A\|
\end{equation*}
Combining with the previous inequality,  $$\sqrt{2}\|A\|\leq \|A\|(1+\lambda).$$
This gives $\lambda \ge \sqrt{2} - 1$.
\end{proof}

The following result provides explicit bounds for the range inclusion constant $\lambda$ under the established orthogonality condition.

We conclude this section, and the paper, by exploring the deep algebraic consequences of the condition $A^*(2A+B)=0$ on range inclusion solutions. We leverage Douglas' factorization theorem to provide explicit lower bounds for the norms of Douglas solutions to the equations $A:B=AX$ and $A=(A+B)X$. Crucially, when the underlying range spaces are closed, these bounds yield elegant, sharp lower estimates involving the Moore-Penrose inverse, illustrating how the geometric alignment of the component operators structurally restricts their relative solution factors.

\begin{prop}
Let $A, B \in \mathcal{B}(\mathcal{H}_1, \mathcal{H}_2)$ be operators such that $$T = \begin{pmatrix} A & A \\ A & A+B \end{pmatrix} \in \mathcal{B}(\mathcal{H}_1 \oplus \mathcal{H}_1, \mathcal{H}_2 \oplus \mathcal{H}_2)$$ is $(M,N)$-complementable and satisfies the condition $A^*(2A+B)=0$.

\begin{enumerate}
\item[(i)] If $X_1$ is the Douglas reduced solution to the equation $A:B = AX_1$, then:
    $
    \|X_1\| \ge \sqrt{2}.
    $\\
    In particular, if $\mathcal{R}(A)$ is closed, then $\|A^\dagger (A:B)\| \ge \sqrt{2}$.

    \item[(ii)] If $X_2$ is the Douglas reduced solution to the equation $A = (A+B)X_2$, then
    $
    \|X_2\| \ge \sqrt{2} - 1.
    $\\
    In particular, if $\mathcal{R}(A+B)$ is closed, then $\|(A+B)^\dagger A\| \ge \sqrt{2} - 1$.
\end{enumerate}
\end{prop}

\begin{proof}
(i) Suppose $X_1$ is the reduced Douglas solution to the equation\\ $A:B=AX_1$, then we have $$\|A:B\| \leq \|A\|\|X_1\|.$$

 Given $A^*(2A+B)=0$, we established $\|A:B\| \ge \sqrt{2}\|A\|$. Comparing these yields $$\|X_1\| \ge \sqrt{2}.$$\\
 If $\mathcal{R}(A)$ is closed, the reduced solution is $X_1 = A^\dagger (A:B)$. Thus, 
$$\|A^\dagger (A:B)\| \ge \sqrt{2}.$$

(ii) Suppose $X_2$ is the Douglas reduced solution to the equation\\ $A = (A+B)X_2$, then
the parallel sum, $A:B = A-AX_2=A(I-X_2).$ From (i), we have $\|(I - X_2)\| \ge \|X_1\| \geq \sqrt{2}$. Applying the triangle inequality, $$\|I\| + \|X_2\| \ge \sqrt{2},$$ which simplifies to $$\|X_2\| \ge \sqrt{2} - 1.$$ If $\mathcal{R}(A+B)$ is closed, the reduced solution is uniquely $X_2 = (A+B)^\dagger A$. Thus, $$\|(A+B)^\dagger A\| \ge \sqrt{2} - 1.$$
\end{proof}

\section*{Acknowledgements}
N/A
\section*{Funding information}
The present work of the second author was partially supported by Anusandhan National Research Foundation (ANRF), Department of Science and Technology, Government of India (Reference Number: MTR/2023/000471) under the scheme “Mathematical Research Impact Centric Support (MATRICS)”.

\section*{Data availability}
Not applicable. No data was used or generated for this study.
\section*{Declaration}
The authors declare that there are no conflicts of interest regarding the publication of this article.

\bibliographystyle{plain}

\end{document}